\documentclass[preprint,11pt,sort&compress,numbers,draft]{elsarticle}

\usepackage{amsmath,amsthm,amsfonts,amssymb,latexsym,mathrsfs,color,cases,url,enumerate}
\usepackage{tikz}

\newtheorem{thm}{Theorem}[section]
\newtheorem{lem}[thm]{Lemma}
\newtheorem{coro}[thm]{Corollary}

\newtheorem{conj}[thm]{Conjecture}

\theoremstyle{definition}

\newtheorem{rem}[thm]{Remark}

\newcommand{\RAC}{\mathbin{\boxplus^{m}_{n}}}

\numberwithin{equation}{section}

\begin{document}
	
	\begin{frontmatter}
		
		\title{The Narayana transformation}
		\author{Jianxi Mao\corref{cor1}}
		\ead{maojx@dlut.edu.cn}
		\cortext[cor1]{Corresponding author.}
		\author{Lijie Wang\corref{cor2}}
		\ead{lijiewang26@hotmail.com}
		\address{School of Mathematical Sciences, Dalian University of Technology, Dalian 116024, P.R. China}
		\date{}
		\begin{abstract}
		For $m\in\mathbb{Z}_{\geq 0}$, let
		\[
		N_{n,m}(x)={}_2F_1(-n,-n-m;m+1;x),
		\]
		which specializes to the Narayana polynomials of types $B$ and $A$ for
		$m=0$ and $m=1$, respectively. We prove that the associated basis
		transformation
		\[
		T_{N_m}\left(\sum_{k=0}^n a_kx^k\right)
		=\sum_{k=0}^n a_kN_{k,m}(x)
		\]
		maps every real-rooted polynomial with nonnegative coefficients to a
		real-rooted polynomial. The proof is based on
		the rectangular additive convolution of polynomials. We then apply
		this result to products of lower triangular matrices and obtain a
		general criterion ensuring that their row generating functions
		remain real-rooted. As consequences, we recover
		this property for powers and products of several classical triangular
		matrices, including Pascal's triangle, the Stirling triangles, and the
		Narayana triangles of types $A$ and $B$. We conclude with conjectures
		concerning the squares of the Eulerian and Delannoy triangles.
		\end{abstract}
		
	\end{frontmatter}
	
	\section{Introduction}
	Real-rooted polynomials frequently arise in combinatorics~\cite{Bra15,Bre89,CYZ,LW07,LY25,YZ17,YZ22} and many other branches of mathematics~\cite{Hua19, MSS15, MSS16, PS76}.
	Since our interest is combinatorial, we are mainly
	concerned with polynomials whose coefficients are nonnegative.
	Liu and Wang~\cite{LW07} presented criteria for proving the real-rootedness of polynomial sequences by recurrence relations. 
	Using these criteria, many well-known combinatorial polynomials are proved to be real-rooted in a unified approach,
	including the row generating functions (RGFs, for short) of the Eulerian triangle, the Narayana triangle, two kinds of Stirling triangles and the Delannoy triangle.
	Another important way to study real-rooted polynomials is to use linear transformations~\cite{Bra07}.
	In this paper, we focus on polynomial basis transformations.
	
	Suppose that $P=\{P_n(x)\}_{n\ge 0}$ is a polynomial sequence with $\deg(P_n(x))=n$ and
	$$
	P_n(x)=\sum_{k=0}^n P(n,k) x^k.
	$$
	Let $\mathbb{R}[x]$ and $\mathbb{R}_{\geq0}[x]$ denote,
	respectively, the sets of polynomials with real coefficients
	and with nonnegative real coefficients. Define the linear
	transformation
	\[
	\mathcal{T}_P:\mathbb{R}[x]\longrightarrow\mathbb{R}[x]
	\]
	by
	\[
	\mathcal{T}_P(x^n)=P_n(x),\qquad n\geq0,
	\]
	and consider its restriction to $\mathbb{R}_{\geq0}[x]$.
	Brenti~\cite{Bre89}  proved that several basis transformations preserve real-rootedness,
    including the cases in which the entries $P(n,k)$ are the unsigned Stirling numbers of the first kind or the Stirling numbers of the second kind. 
	Furthermore, he proposed a conjecture that the Eulerian transformation (i.e., the entries $P(n,k)$ are the Eulerian numbers) preserves real-rootedness.
	Br\"and\'en and Jochemko~\cite{BJ22} studied the Eulerian transformation and disproved this conjecture. 
	Recently, Athanasiadis~\cite{Ath25} presented a criterion under which the polynomial obtained by applying the Eulerian transformation has only real nonpositive roots,
	settling a conjecture posed by Br\"and\'en and Jochemko~\cite{BJ22}.

	The main objective of this paper is to study the Narayana transformation.
	Let 
	\begin{equation}\label{NAB}
		N_A(n,k)
		=\frac{1}{n+1}\binom{n+1}{k}\binom{n+1}{k+1}\quad \textrm{and}\quad
		N_B(n,k)
		=\binom{n}{k}^2,
	\end{equation}
	which are commonly known as the Narayana numbers of types A and B, respectively.
	The $n$-th Narayana polynomials of types A and B are
	\begin{equation*}
		N^A_n(x)=\sum_{k=0}^{n}N_A(n,k)x^k \quad \textrm{and}\quad  N^B_n(x)=\sum_{k=0}^{n}N_B(n,k)x^k
	\end{equation*}
	respectively.
For $m\in \mathbb{Z}_{\ge 0}=\left\{0,1,2,\ldots\right\}$, define
\begin{equation}\label{eqf}
	N_{n,m}(x)={}_2F_1(-n, -n-m;m+1; x):= \sum_{k=0}^{n} \frac{(-n)_k (-n-m)_k}{(m+1)_k} \frac{x^k}{k!}
\end{equation}
where ${}_2F_1$ denotes the Gauss hypergeometric function and
$$
(\alpha)_k = 
\begin{cases}
    	\alpha(\alpha+1)\cdots(\alpha+k-1), & k \in \mathbb{N},\\
    	1, & k = 0,
\end{cases}
$$
is the Pochhammer symbol.
Clearly, by~\eqref{NAB},
$
N^A_n(x)=N_{n,1}(x)$
and 
$
N^B_n(x)=N_{n,0}(x).
$
Define the basis transformation
$$
\mathcal{T}_{N_{m}}(x^n)=N_{n,m}(x).
$$

   	\begin{thm}\label{RZNARAYANA}
Let
$m\in \mathbb{Z}_{\ge 0}.$
 Suppose that
   	$
   	p(x)=\sum_{k=0}^{n}a_kx^k\in\mathbb{R}_{\ge 0}[x]
   	$
   	has only real roots and $\deg(p(x))=n$. 
   	Then the polynomial
   	$$
   	\mathcal{T}_{N_{m}}(p(x))=\sum_{k=0}^{n}a_k\,N_{k,m}(x)
   	$$
   	also has only real nonpositive roots.
\end{thm}
  Taking $m=0$ and $1$,
  Theorem~\ref{RZNARAYANA} implies that both the type A and type B Narayana transformations preserve real-rooted polynomials with nonnegative coefficients.

   	\begin{coro}\label{RZcoro}
   	Suppose that
   	$
   	p(x)=\sum_{k=0}^{n}a_kx^k\in\mathbb{R}_{\ge 0}[x]
   	$
   	has only real roots and $\deg(p(x))=n$. Then the polynomials 
   	$$
   	\sum_{k=0}^{n}a_k\,N^A_{k}(x)\quad \textrm{and} \quad \sum_{k=0}^{n}a_k\,N^B_{k}(x)
   	$$
   	also have only real nonpositive roots.
\end{coro}
	
	Let $\{P_n(x)\}_{n\ge 0}$ be a polynomial sequence with $\deg(P_n(x))=n$.
	Then the coefficient array of the sequence $\{P_n(x)\}_{n\ge 0}$ is therefore a lower triangular matrix.
	A classical result of Mal\'o states that if both $f(x)=\sum_{k=0}^{n} a_kx^k$ and $g(x)=\sum_{k=0}^{n}b_kx^k$ have only real roots and all the roots of $g(x)$ have the same sign, 
	then the polynomial $\sum_{k=0}^{n} a_k b_kx^k$ has only real roots (see~\cite{CC04} and the references therein).
	In particular, 
	if the RGFs of two triangular matrices have only real roots and the entries of one are all nonnegative,
	then the RGFs of their Hadamard product still have only real roots.
	We consider when this property holds under ordinary matrix products.
	
	\begin{thm}\label{rr}
		Let $M=[M(n,k)]_{n,k\ge 0}$ be a lower triangular matrix whose RGFs have only real nonpositive roots.
		Suppose that $B=[B(n,k)]_{n,k\ge 0}$ is a lower triangular matrix and either 
		\begin{itemize}
			\item $B$ is the coefficient matrix of  the sequence $\left\{N_{n,m}(x)\right\}_{n\ge0}$ defined in~\eqref{eqf} with $m\in \mathbb{Z}_{\ge 0}$, or
			\item the entries satisfy the recurrence relation
			\begin{equation}\label{recur}
			B(n,k)=a\, B(n-1,k-1)+(b\,k+c\,(n-1)+d)\, B(n-1,k),
			\end{equation}
			with $a>0,$ $bc=0$, $b,c,d\ge 0$, and the initial condition $B(0,0)=1$,
			where \(B(n,k)=0\) whenever \(k<0\) or \(k>n\).
		\end{itemize}
		Then the RGFs of $MB^{\,r}$ have only real nonpositive roots for $r=1,2,\ldots$.
	\end{thm}
	
	The organization of this paper is as follows.
	In Section~2, we recall several basic properties of the polynomials \(N_{n,m}(x)\) and introduce the rectangular additive convolution of polynomials. 
	Using this convolution, we prove Theorem~\ref{RZNARAYANA}. 
	In Section~3, we interpret polynomial basis transformations in terms of products of lower triangular matrices and prove Theorem~\ref{rr}. 
	Finally, in Section~4, we present concluding remarks and propose an open problem concerning the real-rootedness of the RGFs of the squares of the Eulerian and Delannoy triangles.
	
	\section{Proof of Theorem~\ref{RZNARAYANA}}
	Recall that
\begin{equation*}
		N_{n,m}(x)={}_2F_1(-n, -n-m;m+1; x)=\sum_{k=0}^n N_m(n,k) x^k.
	\end{equation*} 
	Then 
	\begin{align}\label{GM}
		N_m(n,k)=\frac{(-n)_k (-n-m)_k}{(m+1)_k\, k!}=\binom{n}{k}\cdot \frac{(n+m)! m!}{(m+k)! (n+m-k)!}
	\end{align}
	Thus we obtain the symmetry of $N_{n,m}(x)$.
	\begin{lem}\label{sym}
		The polynomial $N_{n,m}(x)$ is symmetric, i.e., $N_m(n,k)=N_m(n,n-k)$.
	\end{lem}

	Dominici, Johnston and Jordaan studied the location of roots of ${}_2F_1(a, b;c; x)$.
	The following result is a special case.
	\begin{lem}(\cite[Theorem 1.1]{DJJ13})\label{DJJ}
		The polynomial $N_{n,m}(x)$
		has only real nonpositive roots for $m\in\mathbb{Z}_{\ge 0}$.
	\end{lem}
	
	We next recall the rectangular additive convolution of Gribinski and Marcus~\cite{GM22}. 
	Let 
	\begin{equation}\label{fg}
	f(x)=\sum_{i=0}^{n}f_i x^{n-i}
	\quad\textrm{and}\quad
	g(x)=\sum_{j=0}^{n}g_j x^{n-j}
	\end{equation}
	with positive leading coefficients $f_0, g_0.$
	Define the rectangular additive convolution of $f(x)$ and $g(x)$ by
	\begin{equation}\label{rac}
		(f\RAC g)(x)
		=\sum_{k=0}^{n}x^{n-k}
		\sum_{i+j=k}\gamma_{i,j}^{(n,m)}f_ig_j,
	\end{equation}
	where
	\begin{equation}\label{gamma}
		\gamma_{i,j}^{(n,m)}
		=\frac{(n-i)!\,(n-j)!}{n!\,(n-i-j)!}
		\frac{(n+m-i)!\,(n+m-j)!}
		{(n+m)!\,(n+m-i-j)!},
	\end{equation}
	see~\cite[Eq. 2.2]{GM22}. 
	Gribinski and Marcus proved that the rectangular additive convolution preserves the property of having only real \textbf{nonnegative} roots.
	\begin{lem}(\cite[Theorem 2.3]{GM22})\label{GM2.3}
    If $f(x)$ and $g(x)$ defined by~\eqref{fg} have only real nonnegative roots
    and the leading coefficients $f_0>0, g_0>0,$ then the rectangular additive convolution $(f\RAC g)(x)$ also has only real nonnegative roots.
	\end{lem}

We now prove Theorem~\ref{RZNARAYANA}.
	
	\begin{proof}[\textbf{Proof of Theorem~\ref{RZNARAYANA}}]
Without loss of generality,
		we may assume that $p(x)$ is monic. Write
		\begin{equation*}
			p(x)=\sum_{k=0}^n a_kx^k=\prod_{i=1}^{n}(x+\lambda_i),
			\qquad \lambda_i\geq 0.
		\end{equation*}
		It suffices to show that
		$$q(x)=\sum_{k=0}^{n}a_k\,N_{k,m}(x)=\sum_{k=0}^{n}a_k\sum_{j=0}^k N_m(k,j) x^j \in\mathbb{R}_{\ge 0}[x]$$
		has only real roots when $m\in\mathbb{Z}_{\ge 0}$.
		Let $R(x)=(-1)^nq(-x).$ Then
		\begin{align*}
			R(x)
			=(-1)^n \sum_{k=0}^{n} a_k \sum_{j=0}^{k}N_m(k,j) (-1)^j x^j
			=\sum_{k=0}^{n} \sum_{j=0}^{k}  (-1)^{n-j} a_k \,N_m(k,j) x^j.
		\end{align*}
		Using the substitutions $r=n-k$ and $\ell=n-j$,
		we have 
		\begin{align}\label{resu}
			R(x)=\sum_{\ell=0}^{n} \sum_{r=0}^{\ell}  (-1)^{\ell} a_{n-r} N_m(n-r,n-\ell) x^{n-\ell}
			=\sum_{\ell=0}^{n}  x^{n-\ell} \sum_{r=0}^{\ell} (-1)^{\ell}  a_{n-r} N_m(n-r,n-\ell).
		\end{align}

		On the other hand, define
		\begin{equation*}
			f(x)=(-1)^n p(-x)=\prod_{i=1}^{n}(x-\lambda_i)
			=\sum_{i=0}^{n}(-1)^ia_{n-i} x^{n-i},
		\end{equation*}
		and
		\begin{align*}
			g(x)=(-1)^nN_{n,m}(-x)
			=\sum_{j=0}^{n}(-1)^j N_m(n,n-j)x^{n-j}=\sum_{j=0}^{n}(-1)^j N_m(n,j)x^{n-j}.
		\end{align*}
		The last equality follows from Lemma~\ref{sym}, i.e., $N_m(n,j)=N_m(n,n-j).$
		Consider the rectangular additive convolution  of $f(x)$ and $g(x)$.
		By~\eqref{rac},
		the coefficient of $x^{n-k}$ in $(f\RAC g)(x)$ is
		$$
		\sum_{i+j=k}\gamma_{i,j}^{(n,m)}f_ig_j=	\sum_{i+j=k}\gamma_{i,j}^{(n,m)}
		(-1)^ia_{n-i}\,(-1)^j N_m(n,j)=\sum_{i+j=k} (-1)^k a_{n-i} \gamma_{i,j}^{(n,m)} N_m(n,j).$$
		If $i+j=k$, then a direct factorial calculation using \eqref{GM} and \eqref{gamma} gives
			\begin{align*}
			\gamma_{i,j}^{(n,m)} N_m(n,j)
			&=\frac{(n-i)!\,(n-j)!}{n!\,(n-k)!}
			\frac{(n+m-i)!\,(n+m-j)!}
			{(n+m)!\,(n+m-k)!}\cdot \frac{n! (n+m)! m!}{j! (m+j)! (n-j)! (n+m-j)!}\\
			&=\frac{(n-i)! (n+m-i)!m!}{j!(m+j)!(n-k)! (n+m-k)!}=N_m(n-i,j).
		\end{align*}
		It follows that
		\begin{align*}
			\sum_{i+j=k}\gamma_{i,j}^{(n,m)}f_ig_j=(-1)^k\sum_{i+j=k}a_{n-i}N_m(n-i,j)
			=(-1)^k\sum_{r=0}^{k}a_{n-r}N_m(n-r,k-r).
		\end{align*}
		By Lemma~\ref{sym}, $N_m(n-r,k-r)=N_m(n-r,n-k)$.
		Comparing with~\eqref{resu} therefore shows that
		$$
			R(x)=\left(f\RAC g\right)(x).
		$$
		
		Clearly, the roots of $f(x)$ and $g(x)$ are nonnegative and the leading coefficients $f_0=g_0=1.$
		Lemma~\ref{GM2.3} implies that $R(x)$ has only real nonnegative roots. 
		Since $$R(x)=(-1)^nq(-x),$$
 the roots of $q(x)$ are all nonpositive.
 This completes the proof.
	\end{proof}
	
	\begin{rem}
		The assumption that \(p(x)\in\mathbb{R}_{\ge 0}[x]\) cannot be omitted in Theorem~\ref{RZNARAYANA}.
		For example, let $p(x)=(x-4)(x-6)=x^2-10x+24$.
		Applying the transformation $\mathcal{T}_{N_1}$, with
		the first three Narayana polynomials $N_{0,1}(x)=1, N_{1,1}(x)=1+x$ and $N_{2,1}(x)=1+3x+x^2,$
		we have 
		$$
		\mathcal{T}_{N_1}(p(x))=N_{2,1}(x)-10N_{1,1}(x)+24 N_{0,1}(x)=x^2-7x+15,
		$$
		which is not real-rooted.
		\end{rem}

	\section{Proof of Theorem~\ref{rr}}
	Let $M=[M(n,k)]_{n,k\ge 0}$ and $B=[B(n,k)]_{n,k\ge 0}$ be two lower triangular matrices
	and $$C=MB=[C(n,k)]_{n,k\ge 0}.$$
	Let
	$$
	M_n(x) =\sum_{k=0}^{n} M(n,k) x^k, \quad B_n(x) =\sum_{k=0}^{n} B(n,k) x^k, \quad C_n(x) =\sum_{k=0}^{n} C(n,k) x^k
	$$
	denote the $n$-th RGFs of $M$, $B$ and $C$, respectively.
	In the first case of Theorem~\ref{rr} , we have $B(n,n)=1$. 
	In the second case, by the recurrence relation~\eqref{recur}, then $B(n,n)=a^n>0.$
	So $\deg(B_n(x))=n$ in both cases.
\begin{lem}\label{mul}
    	We have 
    	$$C_n(x)=\sum_{k=0}^{n} M(n,k) B_k(x).$$
\end{lem}

\begin{proof}
    	A direct computation yields 
    	\begin{align*}
    		C_n(x)&=\sum_{k=0}^{n} C(n,k) x^k=\sum_{k=0}^{n} \sum_{i=k}^{n} M(n,i) B(i,k) x^k\\
    		&=\sum_{i=0}^{n} M(n,i) \sum_{k=0}^{i}  B(i,k) x^k=\sum_{i=0}^{n} M(n,i) B_i(x).
    	\end{align*}
   This proves the result.
\end{proof}
	Suppose that $M=[M(n,k)]_{n,k\ge 0}$ is a lower triangular matrix whose RGFs have only real roots.
	If the basis transformation $\mathcal{T}_B(x^n)=B_n(x)$ preserves real-rootedness, where
    $B_n(x)$ is the $n$-th RGF of $B$ and $\deg(B_n(x))=n$, 
then the RGFs of $MB$ have only real roots.
We next recall some basis transformations that preserve real-rootedness. 

\begin{lem}\cite[Theorem 2.4.2]{Bre89}\label{SST2}
	 	Let
	 	$
	 	f(x)=\sum_{k=0}^{n} a_k \left\langle x\right\rangle_k
	 	$
	 	be a polynomial having only real nonpositive roots,
	 	where $\left\langle x\right\rangle_k=x(x-1)\cdots(x-k+1)$. Then
	 	the polynomial
	 	$
	  \sum_{k=0}^{n} a_k x^k
	 	$
	 	has only real nonpositive roots.
\end{lem}
The Stirling numbers of the second kind $S_2(n,k)$ satisfy the recurrence relation
$$
S_2(n,k)=S_2(n-1,k-1)+k\,S_2(n-1,k), \quad S_2(0,0)=1.
$$
	Let
	$
	S_{n,2}(x)=\sum_{i=0}^{n} S_2(n,i) x^i
	$
	be the \(n\)-th RGF of $[S_2(n,k)]_{n,k\ge0}$.
	By the recurrence relation, 
	the exponential generating function of $\{S_{n,2}(x)\}_{n\ge 0}$ is 
	\begin{equation}\label{egfs2}
	\sum_{n\geq 0}S_{n,2}(x)\frac{z^n}{n!}
	=
	\exp\bigl(x(e^z-1)\bigr),
	\end{equation}
	see, for instance,~\cite{Com74}.
	Note that
	$$
	x^n=\sum_{k=0}^n S_2(n,k) \left\langle x\right\rangle_k.
	$$
	The following corollary follows immediately from Lemma~\ref{SST2}.
	\begin{coro}\label{ST2}
		Let
		$
		f(x)=\sum_{k=0}^{n}a_kx^k\in \mathbb{R}_{\ge 0}[x]
		$
        be a polynomial having only real roots. Then the polynomial
		$
		\sum_{k=0}^{n}
		a_k S_{k,2}(x)
		$
		has only real nonpositive
		roots.
	\end{coro}
	
	Let $\mu>0$. Define
	$$(x|\mu)_n:=x (x+\mu)\cdots(x+(n-1)\mu).$$ 
	Suppose that 
	$(x|\mu)_n=\sum_{k=0}^{n}S_{1,\mu}(n,k)x^k.$
	One obtains the recurrence relation
	$$
	S_{1,\mu}(n,k)=S_{1,\mu}(n-1,k-1)+(n-1) \,\mu\cdot S_{1,\mu}(n-1,k).
	$$
	When $\mu=1$, the entries $S_{1,1}(n,k):=S_1(n,k)$ are the unsigned Stirling numbers of the first kind.
	Brenti~\cite[Theorem 2.4.3]{Bre89} proved that the transformation $x^n\mapsto (x)_n=x(x+1)\ldots(x+n-1)$  preserves real-rootedness.
	Su, Yang and Zhang subsequently generalized this result.

	\begin{lem}\cite[Theorem 4]{Su13}\label{ST1}
	Let
	$
	f(x)=\sum_{k=0}^{n} a_kx^k\in\mathbb{R}_{\ge 0}[x]
	$
	be a polynomial having only real roots. Then the polynomial
	$
	\sum_{k=0}^{n} a_k (x|\mu)_k
	$
	has only simple, real, nonpositive roots for $\mu> 0.$
	\end{lem}

	\begin{proof}[\textbf{Proof of Theorem~\ref{rr}}]
		For \(n\geq 0\), define the RGF of the \(n\)-th row of
		\(B\) by
		\[
		B_n(x)=\sum_{k=0}^n B(n,k)x^k .
		\]
		Multiplying the recurrence relation~\eqref{recur} by \(x^k\) and summing over \(k\),
		we obtain that 
		\begin{equation}\label{rec-rgf}
		B_n(x)
		=
		(ax+c(n-1)+d)B_{n-1}(x)+bxB_{n-1}'(x),
		\qquad B_0(x)=1.
		\end{equation}
Without loss of generality, consider
		$$
		f(x)=\sum_{k=0}^n u_k x^k=\prod_{i=1}^n (x+r_i),
		\qquad
 r_i\geq 0.
		$$
		Define
		$$
		\mathcal{T}_B(f(x))
		=
		\sum_{k=0}^n u_k B_k(x).
		$$
		We now distinguish three cases.
	
		\medskip
		\noindent
		\textbf{Case 1: \(b=c=0\).}
		The recurrence relation~\eqref{rec-rgf} becomes
		\[
		B_n(x)=(ax+d)B_{n-1}(x),
		\]
		and hence
		$
		B_n(x)=(ax+d)^n.
		$
		Therefore,
		\[
		\mathcal{T}_B(f(x))
		=
		\sum_{k=0}^n u_k(ax+d)^k
		=
		f(ax+d).
		\]
		Thus $\mathcal{T}_B(f(x))$ has only real nonpositive roots.

\medskip
\noindent
		\textbf{Case 2: \(b>0\) and \(c=0\).}
		The recurrence relation~\eqref{rec-rgf} becomes
		\begin{equation}\label{Bn}
		B_n(x)
		=
		(ax+d)B_{n-1}(x)+bxB_{n-1}'(x).
		\end{equation}
		Let 
		$
		F(x,z)=\sum_{n\geq 0}B_n(x)\frac{z^n}{n!}
		$
		be the exponential generating function of
		\(\{B_n(x)\}_{n\geq 0}\).
		Multiplying both sides of~\eqref{Bn} by \(z^{n-1}/(n-1)!\) and summing over \(n\geq 1\), we obtain
		$$
		\sum_{n\geq 1}B_n(x)\frac{z^{n-1}}{(n-1)!}
		=
		(ax+d)\sum_{n\geq 1}B_{n-1}(x)\frac{z^{n-1}}{(n-1)!}
		+
		bx\sum_{n\geq 1}B'_{n-1}(x)\frac{z^{n-1}}{(n-1)!}.
		$$
		Since
		$$
		\sum_{n\geq 1}B_n(x)\frac{z^{n-1}}{(n-1)!}
		=
		\frac{\partial F(x,z)}{\partial z}\quad \textrm{and} \quad
		\sum_{n\geq 1}B_{n-1}(x)\frac{z^{n-1}}{(n-1)!}
		=
		F(x,z),
		$$
		while
		$$
		\sum_{n\geq 1}B'_{n-1}(x)\frac{z^{n-1}}{(n-1)!}
		=
		\frac{\partial F(x,z)}{\partial x},
		$$
		it follows that $F(x,z)$ satisfies the partial differential equation
		$$
		\frac{\partial F(x,z)}{\partial z}
		=
		(ax+d)F(x,z)+bx\frac{\partial F(x,z)}{\partial x}.
		$$
	    Solving this equation with \(F(x,0)=1\)
		\begin{equation}\label{rec-rgf-2}
		F(x,z)=\sum_{n\geq 0} B_n(x)\frac{z^n}{n!}
		=
		\exp\left(
		dz+\frac{a}{b}x(e^{bz}-1)
		\right).
      \end{equation}
		Recall from~\eqref{egfs2} that
		\[ \exp\left(\frac{a}{b}x(e^{bz}-1)\right) = \sum_{m\geq 0} S_{m,2}\left(\frac{a}{b}x\right)\frac{(bz)^m}{m!}. \] 
		By~\eqref{rec-rgf-2}, 
		we have 
		\[ F(x,z) = e^{dz} \sum_{m\geq 0} S_{m,2}\left(\frac{a}{b}x\right)\frac{(bz)^m}{m!} = \left(\sum_{r\geq 0}d^r\frac{z^r}{r!}\right) \left(\sum_{m\geq 0}b^mS_{m,2}\left(\frac{a}{b}x\right)\frac{z^m}{m!}\right). \] 
		Comparing the coefficients of \(z^k/k!\) on both sides gives \[ B_k(x) = \sum_{m=0}^{k} \binom{k}{m} d^{k-m}b^m S_{m,2}\left(\frac{a}{b}x\right). \]
		Hence,
		\begin{align*}
		\mathcal{T}_B(f(x))&=\sum_{k=0}^n u_k B_k(x)=\sum_{k=0}^n u_k \sum_{m=0}^k
		\binom{k}{m} d^{\,k-m}b^m
			S_{m,2}\!\left(\frac{a}{b}x\right).
		\end{align*}
	    Since 
		$$
		f(d+bx)
		=
		\sum_{k=0}^{n}u_k(d+bx)^k
		=\sum_{k=0}^n 	u_k  \sum_{m=0}^k
		\binom{k}{m} d^{\,k-m}b^m  x^m
		$$	
		has only real nonpositive roots.
	    By Corollary~\ref{ST2},
	    the polynomial 
	    $$\sum_{k=0}^n 	u_k  \sum_{m=0}^k
	    \binom{k}{m} d^{\,k-m}b^m S_{m,2}\!\left(x\right)
	    $$
	    has only  real nonpositive roots.
	    Replacing $x$ with 
	    $\frac{a}{b}x$ preserves real-rootedness and nonpositivity of the roots.
	    Then we obtain that 
	    $\mathcal{T}_B(f(x))$ has only  real nonpositive roots.

		\medskip
		
		\noindent
		\textbf{Case 3: \(b=0\) and \(c>0\).}
		The recurrence relation~\eqref{rec-rgf} becomes
		\[
		B_n(x)=(ax+c(n-1)+d)B_{n-1}(x).
		\]
		Hence
		\[
		B_n(x)=\prod_{j=0}^{n-1} (ax+cj+d)=\left(ax+d|c\right)_n.
		\]
Then we obtain that
\begin{align*}
	    	\mathcal{T}_B(f(x))=\sum_{k=0}^n u_k B_k(x)=\sum_{k=0}^n u_k 
	    	\left(ax+d|c\right)_k.
\end{align*}
By Lemma~\ref{ST1}, 
the polynomial $\sum_{k=0}^n u_k 
\left(x|c\right)_k$
has only real nonpositive roots.
Taking $x\mapsto ax+d$,
we obtain that 
$
\mathcal{T}_B(f(x))
$
has only real nonpositive roots.

		Combining the three cases, 
		the basis transformation $x^k\mapsto B_k(x)$ preserves the property of having only real nonpositive roots,
		where $B_k(x)$ is the $k$-th RGF of $B$.
		By Theorem~\ref{RZNARAYANA}, the basis transformation $x^k\mapsto N_{k,m}(x)$  preserves real-rootedness.
		Then by Lemma~\ref{mul}, 
if the RGFs of a matrix $M$ have only real nonpositive roots, then
so do the RGFs of $MB$.
Iterating this argument, we obtain that the RGFs of $MB^{\,r}$ have only real nonpositive roots.
This completes the proof.
	\end{proof}

	For the parameter choices 
	$(a,b,c,d)=(1,0,0,1), (1,1,0,0)$ and $(1,0,1,0),$
	the matrix $B$ is, respectively, Pascal's triangle $P$, the Stirling triangle of the second kind $S_2$ and 
	the unsigned Stirling triangle of the first kind $S_1$.
	It is known~\cite{EGS19} that $S_1\cdot S_2$ is the triangle of unsigned Lah numbers.
	By Theorem~\ref{rr},
	the RGFs of the triangle of unsigned Lah numbers have only real roots.
	Moreover, we obtain the following result.

\begin{coro}
If the RGFs of a matrix $M$ have only real nonpositive roots, then 
so do the RGFs of $MB_1B_2\cdots B_s$, where each $B_i$ is one of 
Pascal's triangle, the Stirling triangle of the second kind,
the unsigned Stirling triangle of the first kind, the Narayana  triangle of type $A$ or type $B$.
\end{coro}

	Panzone~\cite{Pan23} proved that for the Stirling triangle of the second kind $S_2$,
    the RGFs of $S_2^2$ and $S_2^3$ have only real nonpositive roots.
	We generalize this result.

	\begin{coro}\label{ke}
	For every $r\in\mathbb{Z}_{\geq0}$, the RGFs of $B^r$ have only real nonpositive roots, where
	$B$ is any one of Pascal's triangle, the Stirling triangle of
	the second kind, the unsigned Stirling triangle of the first
	kind, or a Narayana triangle of type $A$ or type $B$.
	\end{coro}
	
   \begin{rem}
   	The conclusion of Theorem~\ref{rr} does not remain true if the matrix is multiplied on the left.
   	Let $M_0(x)=1, M_1(x)=2+x$ and $M_2(x)=2+3x+x^2$.
   	All three polynomials have only real nonpositive roots.
   	However, $M_0(x)+2M_1(x)+M_2(x)=7+5x+x^2$,
   	whose roots are not real.
   	Thus,
   	left multiplication by Pascal's triangle  does not necessarily preserve the property that all RGFs  have only real nonpositive roots.
   	\end{rem}

	\section{Concluding remarks and open problems}
	
	The Eulerian number $A(n,k)$ counts the number of $n$-permutations with $k-1$ descents. 
	The Eulerian triangle begins as follows
	$$A=[A(n,k)]_{n,k\ge 0}=
	\left[\begin{array}{rrrrrr}
		1 &  &  &  &  &  \\
		0 & 1 &   &   &   &   \\
		0 & 1 & 1 &   &   &   \\
		0 &1 &4 & 1 &   &   \\
		\vdots &  & &  &  \ddots \\
	\end{array}\right].
	$$
The entries satisfy the recurrence relation
	$$
	A(n,k)=(n-k+1)A(n-1,k-1)+kA(n-1,k).
	$$
	Br\"and\'en and Jochemko~\cite{BJ22} showed that the Eulerian transformation does not preserve real-rootedness.
Motivated by Corollary~\ref{ke},
we computed the first 30 RGFs of $A^2$
and found that all of them are real-rooted.
	
	The Delannoy number $D(n,k)$ counts the number of lattice paths from $(0,0)$ to $(n-k,k)$ using steps $(1,0),(0,1)$,
	and $(1,1)$. 
	The Delannoy triangle is of the form
	$$D=[D(n,k)]_{n,k\ge 0}=
	\left[\begin{array}{rrrrrr}
		1 &  &  &  &  &  \\
		1 & 1 &   &   &   &   \\
		1 & 3 & 1 &   &   &   \\
		1 &5 & 5 & 1 &   &   \\
		\vdots &  & &  &  \ddots \\
	\end{array}\right],
	$$
	and the entries satisfy the recurrence relation
	$$
	D(n,k)=D(n-2,k-1)+D(n-1,k-1)+D(n-1,k).
	$$
	Let
	$
	P(x)=(x+3)^3=x^3+9x^2+27x+27.
	$
	Under the Delannoy basis transformation,
	the polynomial
	$$
	\begin{aligned}
		&D_3(x)+9D_2(x)+27D_1(x)+27D_0(x)\\
		=&(1+5x+5x^2+x^3)
		+9(1+3x+x^2)
		+27(1+x)+27\\
		=&x^3+14x^2+59x+64
	\end{aligned}
	$$
	does not have only real roots.
    It is natural to ask under what conditions the Delannoy transformation preserves real-rootedness.
	
	A proper Riordan array, 
	denoted by $(g(t),f(t))$, 
	is an infinite lower triangular matrix whose ordinary generating function of the $k$-th column is $g(t)(f(t))^k$, 
	where $g(0) \ne 0$, $f(0) = 0$ and $f'(0) \ne 0$. 
	The set of all Riordan arrays forms the Riordan group under matrix multiplication,
	\begin{equation}\label{FTRA}
		(g(t),f(t))(d(t),h(t))=(g(t)d(f(t)),h(f(t))),
	\end{equation}
	see, for instance,~\cite{SSB22}.
	It is known~\cite{CW19} that the Delannoy triangle $D$ is a Riordan array 
\[
D=\left(\frac{1}{1-t},\frac{t(1+t)}{1-t}\right).
\]

Let
$
G=D^2.
$
By~\eqref{FTRA},
we have
\[
G=
\left(
\frac{1}{1-2t-t^2},
\frac{t(1+t)(1+t^2)}
{(1-t)(1-2t-t^2)}
\right).
\]
Define the $n$-th RGF of $G$ by
$
G_n(x)=\sum_{k=0}^{n}G(n,k)x^k.
$
Thus, its bivariate generating function is
\begin{align*}
F(x,t)=\sum_{n\ge 0}G_n(x)t^n
&=
\frac{\frac{1}{1-2t-t^2}}
{1-x\frac{t(1+t)(1+t^2)}
    	{(1-t)(1-2t-t^2)}}
=
\frac{1-t}
{(1-t)(1-2t-t^2)-xt(1+t)(1+t^2)}\\
&=\frac{1-t}
{1-(x+3)t+(1-x)t^2+(1-x)t^3-xt^4}.
\end{align*}
Hence,
$$
\left(
1-(x+3)t+(1-x)t^2+(1-x)t^3-xt^4
\right)F(x,t)
=
1-t.
$$
Equating coefficients of  $t^n$, we obtain that
$$
G_n(x)=(x+3)G_{n-1}(x)
    	+(x-1)G_{n-2}(x)
    	+(x-1)G_{n-3}(x)
    	+xG_{n-4}(x)
$$
for $n\ge 4$, with the initial conditions
$$
G_0(x)=1, G_1(x)=x+2, G_2(x)=x^2+6x+5, G_3(x)=x^3+10x^2+25x+12.
$$
Numerical evidence suggests that the RGFs of $D^2$ have only real nonpositive roots.

\begin{conj}
    	The RGFs of $A^2$ and $D^2$ have only real nonpositive roots.
\end{conj}
	
	\section*{Acknowledgements}
	
	This work was partially supported by the National Natural Science Foundation of China
	(Grant No. 12201100).

\end{document}